\documentclass[12pt]{article}
\usepackage[centertags]{amsmath}
\usepackage{amsfonts}
\usepackage{amssymb}
\usepackage{amsthm}
\DeclareMathAlphabet{\mathpzc}{OT1}{pzc}{m}{it}

\theoremstyle{plain}
\newtheorem{thm}{Theorem}[section]
\newtheorem{cor}[thm]{Corollary}
\newtheorem{lem}[thm]{Lemma}

\theoremstyle{definition}
\newtheorem{defn}[thm]{Definition}

\newtheorem{rem}[thm]{Remark}
\theoremstyle{remark}
\numberwithin{equation}{section}

\newcommand{\beast}{\begin{eqnarray*}}
\newcommand{\eeast}{\end{eqnarray*}}
\title{A family of group divisible designs with arbitrary block sizes}

\author{Yu-pei Huang\footnote{E-mail address: yphuang@bnuz.edu.cn, College of Applied Mathematics, Beijing Normal University, Zhuhai 519087, P.R. China.}
\and
Chia-an Liu\footnote{Corresponding author. E-mail address: liuchiaan8@gmail.com, Department of Mathematical Sciences, University of Delaware Newark, Delaware 19716, U.S.A.}
\and
Yaotsu Chang \footnote{E-mail address: ytchang@cloud.isu.edu.tw, Department of Financial and Computational Mathematics, I-Shou University, Taiwan 84001, R.O.C.}
\and
Chong-Dao Lee\footnote{E-mail address: chongdao@cloud.isu.edu.tw, Department of Communication Engineering, I-Shou University, Taiwan 84001, R.O.C.}}

\date{September 4, 2018}
\begin{document}
\maketitle

\bibliographystyle{plain}

\bigskip

\begin{abstract}
Recently, a construction of group divisible designs (GDDs) derived from the decoding of quadratic residue (QR) codes was given. In this paper, we extend the idea to obtain a new family of GDDs, which is also involved with a well-known balanced incomplete block design (BIBD).\\\\
%which has not been previously known.\\
%
\end{abstract}

%%%%%%%%%%%%%%%%   1.Introduction.  %%%%%%%%%%%%%%%%%%%%%%%%%%%

\section{Introduction}      \label{sec_introduction}

Combinatorial designs and the theory of error-correcting codes are two research topics which are closely related.
Assmus and Mattson in 1969 \cite{am:69} first proposed the relationship between balanced incomplete block designs (BIBDs) and error-correcting codes. For instance, the codewords of any fixed weight in an extended quadratic residue code \cite{am:69} form a $2$-design.
Later, BIBDs can also be constructed from
Reed-Muller codes \cite{cgl:08}, extremal binary doubly-even self-dual codes \cite{cgl:08}, and Pless symmetry codes \cite{p:72}.

\medskip

Quadratic residue (QR) codes generated by irreducible polynomials are called Type I QR codes, and those generated by reducible polynomials are Type II.
In 2003, Chang et al.~\cite{ctrcl:03} developed algebraic decoding of three Type I binary QR codes.
For Type I QR codes, if the first syndrome is zero then one can assume that there is no error occurred. However, for Type II QR codes, one cannot suppose that the error pattern is zero, i.e., no error occurred, even if the first syndrome is zero.
Motivated by the decoding of QR codes, Lee et al.~\cite{lcc:17}
provided a construction of group divisible designs.
They investigated the collection of all error patterns of weight three for the Type II QR code of length $31$ which is with zero first syndromes and found some combinatorial structure. A new family of GDDs with block sizes $3$ to $7$ was given and further generalized by Ji~\cite{j:17} with arbitrary block sizes on finite fields.

\medskip

This research is a sequel of~\cite{lcc:17}. The authors in~\cite{lcc:17,j:17} considered the error patterns $(x_1,x_2,\ldots,x_k)$ satisfying the equation $\gamma^{x_1}+\gamma^{x_2}+\cdots+\gamma^{x_k}=1\in\mathbb{F}_{2^m}$ with no proper subset $S$ of $\{x_1,x_2,\ldots,x_k\}$ such that $\sum_{i\in S}\gamma^i=1$, where distinct integers $1\leq x_i\leq 2^m-2$ for $1\leq i\leq k\leq m$ and $\gamma$ is a primitive element of the binary extension field $\mathbb{F}_{2^m}$. While $k=2,$ those error patterns form a group set $\mathcal{G}$. In this study, we propose another construction of GDDs by assuming the sum of each error pattern to be any prescribed nonzero element $\alpha$ instead of $1,$ and omitting the constraints for the sum of proper subset $S$ of $\{x_1,x_2,\ldots,x_k\}.$
One may notice that these new GDDs are similar to the previous one~\cite{lcc:17} when $k$ is $3$ or $4,$ but the divergence appears for $k\geq5.$

\medskip

The paper is organized as follows.
To study the new family of GDDs, a construction of BIBDs related to the Hamming code is provided in Section~\ref{sec_BIBD}.
The details of our methods to construct GDDs are depicted in Section~\ref{sec_GDD}.
A short conclusion is given in the last section.

\bigskip

%%%%%%%%%%%%%%%%%%%%%%%%%%%%%%%%%%%
%%%%%%%%%%% Chapter 2 %%%%%%%%%%%%%
%%%%%%%%%%%%%%%%%%%%%%%%%%%%%%%%%%%
\section{A construction of balanced incomplete block designs}           \label{sec_BIBD}
This section is composed of two subsections. The first subsection describes a brief review of BIBDs.
The second subsection introduces a family of BIBDs and shows their balance parameters.
\subsection{Basic results and notations}
\begin{defn}    \cite[Definition~1.2]{s:04}
Let $v,k$, and $\lambda$ be positive integers such that $v>k \geq 2$. A \emph{balanced incomplete block design} $(v,k, \lambda)$-BIBD is a pair $(X, \mathcal{B})$ such that the following properties are satisfied:
\begin{enumerate}
\item[(i)] $X$ is a set of elements called {\it points} with cardinality $|X|=v,$
\item[(ii)] $\mathcal{B}$ is a class of nonempty $k$-subsets of $X$ called {\it blocks}, and
\item[(iii)] every pair of distinct points is contained in exactly $\lambda$ blocks.
\end{enumerate}
Particularly, (iii) is called the \emph{balance property} and $\lambda$ is called the \emph{balance parameter} of $(X,\mathcal{B}).$
\end{defn}

\medskip

There are several parameters in a BIBD which are described in the following.
\begin{thm}     \cite[Theorem~1.9]{s:04}    \label{thm_basic}
Let $(X,\mathcal{B})$ be a $(v,k,\lambda)$-BIBD. Then every point occurs in exactly
$$r=\frac{\lambda(v-1)}{k-1}$$
blocks, and the number of blocks
$$b=|\mathcal{B}|= \frac{vr}{k}.$$
\end{thm}

\medskip

Let $m\geq 3$ be a positive integer and $\mathbb{F}_{2^m}$ be the finite field of order $2^m.$ Then the multiplicative group $\mathbb{F}_{2^m}^*=\mathbb{F}_{2^m}\setminus\{0\}$ is cyclic of order $2^m-1,$ where $0$ is the zero element of $\mathbb{F}_{2^m}.$ The following definition gives sets of blocks in which the sum of elements is $0.$ The ideas of zero-sum blocks for the construction of BIBDs are also studied in~\cite{s:13,s:16}.
\begin{defn}
For each integer $k$ with $3 \leq k \leq 2^m-4,$ let
\begin{equation}
W_{k}=\big\{B \subseteq \mathbb{F}_{2^m}^*~\mid~|B|=k~\text{and}~\sum_{i\in B}i=0\big\}
\nonumber
\end{equation}
be the collection of $k$-subsets of $\mathbb{F}_{2^m}^*$ in which the sum of elements is zero.
\end{defn}

\medskip

It should be noticed that the $k$-sets of nonzero elements summing up to zero in the Galois field with $2^m$ elements can be seen as codewords of weight $k$ in the $(2^m-1,2^m-m-1,3)$ Hamming code. According to \cite[p. 129]{M:77} and \cite{S:80}, the number $b_k$ of such codewords can be determined recursively from the relation $(k+1)b_{k+1}+b_k+(v-k+1)b_{k-1}=\binom{v}{k}$, where $v=2^m-1$, and a closed-form expression for the number $b_k$ is given in \cite[Proposition 4.1]{E:94}.

\medskip

It is not hard to show that $W_k$ is nonempty for every $3 \leq k \leq 2^m-4$ by induction. First, for distinct $i,j\in\mathbb{F}_{2^m}^*$ there exists a block $\{i,j,i+j\}\in W_3.$ Suppose that for $4\leq k\leq 2^{m-1}-1$ there exists a $(k-1)$-subset $B_0\in W_{k-1}.$
We will use $B_0$ to construct a $k$-subset $\tilde{B}_0$ of $\mathbb{F}_{2^m}^*$ in which the sum of elements is still zero.
Let $\alpha$ be an element in $B_0.$ We define
\begin{equation}
\mathpzc{H}_{\alpha}=\mathbb{F}_{2^m}/\{0,\alpha\}=\big\{\{x,x+\alpha\}~|~x\in\mathbb{F}_{2^m}\big\},
\nonumber
\end{equation}
and give some background information of $\mathpzc{H}_{\alpha}$ in the following.

\begin{rem}     \label{rem_H}
Consider the additive group $\langle \mathbb{F}_{2^m}, +\rangle$. For some $\alpha \in \mathbb{F}_{2^m}^*,$ since $\mathbb{F}_{2^m}$ has characteristic $2,$ one has that $\{0,\alpha\}$ is a subgroup of $\langle\mathbb{F}_{2^m},+\rangle.$ Hence,
%$$\mathpzc{H}_{\alpha}=\mathbb{F}_{2^m}/\{0,\alpha\}=\big\{\{x,x+\alpha\}~|~
%x\in\mathbb{F}_{2^m}\big\}$$
$\mathpzc{H}_{\alpha}$ is well-defined and forms a partition of $\mathbb{F}_{2^m}$ with cardinality $|\mathpzc{H}_{\alpha}|=2^{m-1}$.
\end{rem}

\medskip

Since $|\mathpzc{H}_{\alpha}\setminus\{\{0,\alpha\}\}|=2^{m-1}-1>k-1,$ by Pigeonhole Principle there exists $x_0\in\mathbb{F}_{2^m}\setminus\{0,\alpha\}$ such that $\{x_0,x_0+\alpha\}\cap B_0=\phi.$ Then one has a $k$-subset $\tilde{B_0}=B_0\setminus\{\alpha\}\cup\{x_0,x_0+\alpha\}$ of $\mathbb{F}_{2^m}^*.$ Note that $\sum_{i\in\tilde{B_0}}i=\sum_{i\in B_0}i=0$ and hence $\tilde{B_0}\in W_k.$ Now, $W_k$ is nonempty for $3\leq k\leq 2^{m-1}-1.$ Since the sum of elements in $\mathbb{F}_{2^m}^*$ is zero, $B\in W_k$ if and only if $\mathbb{F}_{2^m}^*\setminus B\in W_{2^m-1-k},$ and the proof is completed. Moreover, the fact
\begin{equation}
|W_k|=|W_{2^m-1-k}|~~~\text{for}~3 \leq k \leq 2^m-4
\nonumber
\end{equation}
immediately follows.

\medskip

The set $W_k$ will play an important role in constructing BIBDs as illustrated in the next subsection.

\medskip

\subsection{BIBDs and their balance parameters}    \label{subsec_proposedBIBD}
The aim of this subsection is to prove Theorem~\ref{thm_BIBD} which states that $(\mathbb{F}_{2^m}^*,W_k)$ is a $(2^{m}-1,k,\lambda_k)$-BIBD for $3\leq k\leq 2^m-4.$
Then the balance parameters $\lambda_k$ are given in Corollary~\ref{cor_lambda_k}.

\medskip

Let $\mathpzc{H}_\alpha$ be ordered by some one-to-one mapping
$$\mathcal{O}_\alpha:~\mathpzc{H}_\alpha\rightarrow \{1,2,\ldots,2^{m-1}\}.$$
\begin{defn}        \label{defn_representative}
Given $B\subseteq \mathbb{F}_{2^m}^*,$ if $B\setminus(B+\alpha)$ is nonempty, then there exists a unique $\beta\in B\setminus(B+\alpha)$ with the maximal ordering in $\mathcal{O}_\alpha$, i.e.,
$$\mathcal{O}_\alpha(\{\beta,\beta+\alpha\})=\max_{\gamma\in B\setminus(B+\alpha)} \mathcal{O}_\alpha(\{\gamma,\gamma+\alpha\}).$$
We call $\beta$ the \emph{representative} of $B$ with respect to $\mathcal{O}_\alpha.$
\end{defn}

Note that if $\sum_{i\in B}i\notin\{0,\alpha\}$ then $B\setminus(B+\alpha)$ is nonempty, which provides a sufficient condition for the existence of the representative $\beta\in B.$

\medskip

For $3\leq k\leq 2^m-4$ and distinct $i,j\in \mathbb{F}_{2^m}^*,$ let
$$W_k^{i,j}=\{B\in W_k~\mid~i,j\in B\}$$
be the set of blocks in $W_k$ that contains $i,j.$ Note that $W_k^{i,j}$ is finite since it is a subset of $W_k.$ We study the cardinality of $W_k^{i,j}$ in the following.
\begin{lem}         \label{lem_eqsize}
For distinct $i,j,\ell \in \mathbb{F}_{2^m}^*,$ $|W_k^{i,j}|=|W_k^{i,\ell}|.$
\end{lem}
\begin{proof}
Let $\alpha=j+\ell$ and $\mathpzc{H}_{\alpha}=\big\{\{x,x+\alpha\}\mid x\in\mathbb{F}_{2^m}\big\}$ be ordered by some one-to-one mapping $\mathcal{O}_\alpha:~\mathpzc{H}_\alpha\rightarrow \{1,2,\ldots,2^{m-1}\}.$ Define a function $\phi:~W_k^{i,j}\rightarrow W_k^{i,\ell}$ as
\begin{equation}
\phi(B)= \left\{
\begin{array}{ll}B, &~\mbox{if}~\ell \in B\\
B\setminus \{j,\beta\} \cup \{\ell,\beta+\alpha\}, &~\mbox{if}~\ell \notin B
\end{array}\right.
\nonumber
\end{equation}
for each $B\in W_k^{i,j},$ where $\beta$ is the representative of $B^-=B\setminus\{i,j,\alpha\}$ with respect to $\mathcal{O}_\alpha.$ Since the sum of elements in $B^-$ is $i+j$ or $i+\ell$ (which is not in $\{0,\alpha\}$), the set $B^-\setminus (B^- +\alpha)$ is nonempty and the mapping $\phi$ is well-defined.

\medskip

Claim that $\phi$ is a bijection. Define another function $\tilde{\phi}:~W_k^{i,\ell}\rightarrow W_k^{i,j}$ as
\begin{equation}
\tilde{\phi}(\tilde{B})= \left\{
\begin{array}{ll}\tilde{B}, & ~\mbox{if}~j \in \tilde{B}\\
\tilde{B}\setminus \{\ell,\tilde{\beta}\} \cup \{j,\tilde{\beta}+\alpha\}, & ~\mbox{if}~j \notin \tilde{B}
\end{array}\right.
\nonumber
\end{equation}
for each $\tilde{B}\in W_k^{i,\ell},$ where $\tilde{\beta}$ is the representative of $\tilde{B}^-=\tilde{B}\setminus\{i,\ell,\alpha\}$ with respect to $\mathcal{O}_\alpha.$ Similarly, the mapping $\tilde{\phi}$ is well-defined since the sum of elements in $\tilde{B}^-$ is $i+\ell$ or $i+j$ (which is not in $\{0,\alpha\}$). It is clear that $\tilde{\phi}(\phi(B))=B$ if $B\in W_k^{i,j}$ with $\ell\in B.$ On the other hand, for every $B\in W_k^{i,j}$ with $\ell\notin B,$ one can observe that $\beta$ is the representative of $B^-$ with respect to $\mathcal{O}_\alpha$ if and only if $\tilde{\beta}=\beta+\alpha$ is the representative of $\tilde{B}^-$ with respect to $\mathcal{O}_\alpha,$ where $\tilde{B}=B\setminus\{j,\beta\}\cup\{\ell,\tilde{\beta}\}.$ Therefore, $\tilde{\phi}(\phi(B))=B$ if $B\in W_k^{i,j}$ with $\ell\notin B.$ Consequently, $\phi$ is a bijection from $W_k^{i,j}$ to $W_k^{i,\ell}$ with the inverse $\tilde{\phi},$ and the result follows.
\end{proof}

\medskip

\begin{thm}     \label{thm_BIBD}
For each integer $k$ with $3 \leq k \leq 2^m-4,$ the pair $(\mathbb{F}_{2^m}^*,W_k)$ is a $(2^{m}-1, k, \lambda_k)$-BIBD.
\end{thm}
\begin{proof}
Let $h,i,j,\ell$ be distinct elements in $\mathbb{F}_{2^m}^*$. By Lemma~\ref{lem_eqsize}, one has
$$|W_k^{h,i}|=|W_k^{i,j}|=|W_k^{j,\ell}|$$
for $3\leq k \leq 2^m-4.$ Thus, the balance property for being a BIBD is confirmed. That is, $(\mathbb{F}_{2^m}^*,W_k)$ is a $(2^m-1,k,\lambda_k)$-BIBD for some constant $\lambda_k=|W_k^{i,j}|.$
\end{proof}

\medskip

Theorem~\ref{thm_BIBD} indicates that for two positive integers $k,m$ with $3\leq k\leq 2^m-4$ the pair $(\mathbb{F}_{2^m}^*,W_k)$ is a $(2^m-1,k,\lambda_k)$-BIBD, which is proved above. Then the remainder of this subsection is to show that the
balance parameter $\lambda_k$ is obtained in recursive relations. The method we use is basically by counting. For some element $\alpha\in\mathbb{F}_{2^m}^*,$ the numbers of blocks involved with $\alpha$ are given below.
\begin{lem}         \label{lem_IJreccurence}
For $3\leq k\leq 2^m-4$ and some element $\alpha\in\mathbb{F}_{2^m}^*,$ let
$$I_k^\alpha=\{B\subseteq \mathbb{F}_{2^m}\setminus\{0,\alpha\}~\mid~|B|=k~\text{and}~\sum_{i\in B}i=\alpha\}$$
and
$$J_k^\alpha=\{B\subseteq \mathbb{F}_{2^m}\setminus\{0,\alpha\}~\mid~|B|=k~\text{and}~\sum_{i\in B}i=0\}.$$
Then
\begin{equation}
|I_k^\alpha|=\left\{\begin{array}{ll}
                        |J_k^\alpha|, & \text{if}~k\equiv 1,3~(\text{mod}~4) \\
                        |J_k^\alpha|+{2^{m-1}-1 \choose k/2}, & \text{if}~k\equiv 2~(\text{mod}~4) \\
                        |J_k^\alpha|-{2^{m-1}-1 \choose k/2}, & \text{if}~k\equiv 0~(\text{mod}~4).
                      \end{array}\right.
\nonumber
\end{equation}
\end{lem}
\begin{proof}
We will prove this result by the mappings between $I_k^\alpha$ and $J_k^\alpha.$
This proof can be divided into three cases.

\medskip

\textbf{Case 1:} $k\equiv 1,3~(\text{mod}~4).$ Since $k$ is odd, for each $B\in I_k^\alpha$ we have $B\setminus(B+\alpha)$ is nonempty. Hence, there exists $\beta\in B$ such that $\beta$ is the representative of $B$ with respect to some proper ordering $\mathcal{O}_\alpha$ of $\mathpzc{H}_\alpha.$ In this case, the mapping $\phi: I_k^\alpha \rightarrow J_k^\alpha$ defined by
$$\phi(B)=B\setminus\{\beta\}\cup\{\beta+\alpha\}$$
is a bijection. Therefore, one has $|I_k^\alpha|=|J_k^\alpha|.$

\medskip

In the following argument, let
$$L_k^\alpha=\{B\subseteq \mathbb{F}_{2^m}^*~\mid~|B|=k~\text{and}~B=B+\alpha\}$$
for even $k.$

\medskip

\textbf{Case 2:} $k\equiv 2~(\text{mod}~4).$ In this case, $k/2$ is odd, and every $B\in L_k^\alpha$ is with $\sum_{i\in B}i=\alpha.$ Hence, $L_k^\alpha\subseteq I_k^\alpha.$ Besides, since $B\setminus(B+\alpha)$ is nonempty for each $B\in I_k^\alpha\setminus L_k^\alpha,$ one has the representative $\beta$ of $B$ with respect to $\mathcal{O}_\alpha.$ Consequently, the mapping $\phi:~I_k^\alpha\setminus L_k^\alpha \rightarrow J_k^\alpha$ is also a bijection, and thus $|I_k^\alpha|-|L_k^\alpha|=|J_k^\alpha|.$ Moreover, we can see that
$|L_k^\alpha|={2^{m-1}-1 \choose k/2}$ because $|\mathpzc{H}_\alpha\setminus\{\{0,\alpha\}\}|=2^{m-1}-1.$ The equality in case 2 follows.

\medskip

\textbf{Case 3:} $k\equiv 0~(\text{mod}~4).$ Similarly, since $k/2$ is even, one has the bijection $\phi:~I_k^\alpha \rightarrow J_k^\alpha\setminus L_k^\alpha.$ Therefore, $|I_k^\alpha|=|J_k^\alpha|-|L_k^\alpha|,$ and the proof is completed.
\end{proof}

\medskip

For $k\geq 3,$ $(\mathbb{F}_{2^m}^*,W_k)$ is a BIBD by Theorem~\ref{thm_BIBD}. Let $b_k=|W_k|$ be the number of blocks and $r_k$ denote the number of blocks in which each point occurs. The following result is helpful to evaluate the values of those parameters.
\begin{thm}         \label{thm_BIBDrecurrence}
For $3\leq k \leq 2^m-4,$ there are the following recurrence relations
\begin{equation}
r_{k+1}=\left\{\begin{array}{ll}
                        b_k-r_k, & \text{if}~k\equiv 1,3~(\text{mod}~4) \\
                        b_k-r_k+{2^{m-1}-1 \choose k/2}, & \text{if}~k\equiv 2~(\text{mod}~4) \\
                        b_k-r_k-{2^{m-1}-1 \choose k/2}, & \text{if}~k\equiv 0~(\text{mod}~4)
                      \end{array}\right.
\nonumber
\end{equation}
where $r_{2^m-3}:=0.$
\end{thm}
\begin{proof}
We prove it by counting the values of $|I_k^\alpha|$ and $|J_k^\alpha|$ defined in Lemma~\ref{lem_IJreccurence}.
From definition, we can observe that
\begin{eqnarray}
|I_k^\alpha|&=&\big|\{B\subseteq \mathbb{F}_{2^m}\setminus\{0,\alpha\}~\mid~|B|=k~\text{and}~\sum_{i\in B}i=\alpha\}\big|
\nonumber   \\
&=&\big|\{\tilde{B}\subseteq \mathbb{F}_{2^m}^*~\mid~|B|=k+1~\text{and}~\sum_{i\in \tilde{B}}i=0\}\big|
\nonumber   \\
&=&r_{k+1}
\nonumber
\end{eqnarray}
by letting $\tilde{B}=B\cup\{\alpha\}$ for each $B\in I_k^\alpha.$ On the other hand,
\begin{eqnarray}
|J_k^\alpha|&=&\big|\{B\subseteq \mathbb{F}_{2^m}\setminus\{0,\alpha\}~\mid~|B|=k~\text{and}~\sum_{i\in B}i=0\}\big|
\nonumber   \\
&=&\big|\{B\subseteq \mathbb{F}_{2^m}^*~\mid~|B|=k~\text{and}~\sum_{i\in B}i=0\}\big|
\nonumber   \\
&&-\big|\{B\subseteq \mathbb{F}_{2^m}^*~\text{with}~\alpha\in B~\mid~|B|=k~\text{and}~\sum_{i\in B}i=0\}\big|
\nonumber   \\
&=&b_k-r_k
\nonumber
\end{eqnarray}
by applying the principle of inclusion and exclusion. The result directly follows from Lemma~\ref{lem_IJreccurence}.
\end{proof}

\medskip

The initial conditions of Theorem~\ref{thm_BIBDrecurrence} are provided as follows.
\begin{rem}     \label{rem_BIBD_ini_condi}
It is clear that $\lambda_3=1,$ since there exists a unique block $\{i,j,i+j\}\in W_3$ for any two distinct elements $i,j\in \mathbb{F}_{2^m}^*$.
Then by Theorem~\ref{thm_basic} one has
$$r_3=\frac{2^m-2}{2}~~~\text{and}~~~ b_3=\frac{(2^m-1)(2^m-2)}{3!}.$$
Actually, while $k=2,$ it is straightforward to define $b_2=r_2=\lambda_2=0$ because there are no blocks in $W_2.$ The recurrence formula in Theorem~\ref{thm_BIBDrecurrence} also indicates that $r_3={2^{m-1}-1 \choose 1}=(2^m-2)/2.$
\end{rem}

\medskip

Now, the recurrence relations of balance parameters $\lambda_k$ are presented in the following which is directly from Theorem~\ref{thm_BIBDrecurrence} and Theorem~\ref{thm_basic}.
\begin{cor}     \label{cor_lambda_k}
For $3\leq k\leq 2^m-4,$
\begin{equation}
\lambda_{k+1}=\left\{\begin{array}{ll}
                        \frac{2^m-k-1}{k-1}\lambda_k, & \text{if}~k\equiv 1,3~(\text{mod}~4) \\
                        \frac{2^m-k-1}{k-1}\lambda_k+{2^{m-1}-2 \choose k/2-1}, & \text{if}~k\equiv 2~(\text{mod}~4) \\
                        \frac{2^m-k-1}{k-1}\lambda_k-{2^{m-1}-2 \choose k/2-1}, & \text{if}~k\equiv 0~(\text{mod}~4)
                      \end{array}\right.
\nonumber
\end{equation}
where $\lambda_{2^m-3}:=0.$
In one formula,
\begin{equation}
\lambda_{k+1}=\frac{2^m-k-1}{k-1}\lambda_k-\cos\frac{k\pi}{2}{2^{m-1}-2 \choose \lfloor k/2-1\rfloor}.
\nonumber
\end{equation}
\qed
\end{cor}

\medskip

Based on the above results, the parameters $\lambda_k$ with $3\leq k\leq 7$ are listed in Table 1 for some $m\geq 4.$
%\extrarowheight=5pt
\renewcommand{\arraystretch}{1.5}
\begin{center}
{\bf Table 1:} The balance parameter $\lambda_k$ of the BIBD $(\mathbb{F}_{2^m}^*,W_k)$ for $3\leq k\leq 7$.

\bigskip
\begin{tabular}{|c|c|}
\hline
      & $\lambda_{k}$ \\  \hline
$k=3$ & $1$      \\ \hline
$k=4$ & $\frac{2^m-4}{2}$  \\ \hline
$k=5$ & $\frac{(2^m-4)(2^m-8)}{3!}$  \\ \hline
$k=6$ & $\frac{(2^m-4)(2^m-6)(2^m-8)}{4!}$  \\ \hline
$k=7$ & $\frac{(2^m-4)(2^m-6)(2^{2m}-15\cdot 2^m+71)}{5!}$ \\ \hline
\end{tabular}\\
\end{center}

As a consequence of Theorem~\ref{thm_BIBD} and Corollary~\ref{cor_lambda_k},
the parameters $(v,k,\lambda_k)$ of BIBDs with small block sizes are listed below: $(7,3,1),$ $(15,3,1),$ $(31,3,1),$ $(7,4,2),$ $(15,4,6),$ $(31,4,14),$ $(15,5,16),$ $(31,5,112),$ $(15,6,40),$ and $(15,7,87).$

\medskip

A series of BIBDs obtained in Theorem~\ref{thm_BIBD} will be used to construct a new family of GDDs as shown in the next section.

\bigskip

%%%%%%%%%%%%%%%%%%%%%%%%%%%%%%%%%%%%%%%%
%%%%%%%%%%%    Section GDD   %%%%%%%%%%%
%%%%%%%%%%%%%%%%%%%%%%%%%%%%%%%%%%%%%%%%
\section{A construction of group divisible designs}     \label{sec_GDD}
This section consists of two subsections.
Section~\ref{subsec_GDDnotations} gives the definition of a GDD.
Section~\ref{subsec_proposedGDD} is the main result of this paper, which presents new GDDs with arbitrary block sizes.

\subsection{Notations}      \label{subsec_GDDnotations}
GDD is a topic generalized from the pairwise balanced design (well-known as PBD) \cite[p. 231]{cd:07}.
Since GDD has been widely applied to graphs \cite{fr:98} and matrices \cite{ss:98}, many authors proposed different constructions of a GDD. One can see \cite{fr:98,ss:98,hs:04}, \cite[Definition~1.4.2]{a:90}, \cite[Definition~7.14]{s:04} and \cite[Definition~5.5]{w:09} for some examples.
The definition of a GDD is as follows.
\begin{defn}    \cite[p. 231]{cd:07}    \label{defn_GDD}
Let $k$ and $\lambda$ be positive integers. A {\it group divisible design} $(k,\lambda)$-GDD is a triple $(X,\mathcal{G}, \mathcal{B})$, where $X$ is a finite set of cardinality $v$, $\mathcal{G}$ is a partition of $X$ into \emph{groups}, and $\mathcal{B}$ is a family of subsets ({\it blocks}) of $X$ that satisfy
\begin{enumerate}
\item[(i)] if $B \in \mathcal{B}$ then $|B|=k$,
\item[(ii)] every pair of distinct elements of $X$ occurs in exactly $\lambda$ blocks or one group, but not both, and
\item[(iii)] $|\mathcal{G}|>1.$
\end{enumerate}
In particular, (ii) is called the balance property and $\lambda$ is called the balance parameter of $(X,\mathcal{G}, \mathcal{B}).$
\end{defn}

\medskip

\subsection{Proposed GDDs}      \label{subsec_proposedGDD}
Throughout this subsection, let $\alpha$ be an element in $\mathbb{F}_{2^{m+1}}^*$ and
$V_\alpha=\mathbb{F}_{2^{m+1}}\setminus\{0,\alpha\}.$
Consider the collection $U_{\alpha,2}$ of some $2$-subsets of $V_\alpha$ such that
$$U_{\alpha,2}=\big\{\{i,j\}\subseteq V_\alpha~|~i+j=\alpha\big\}.$$
Furthermore, for each $3\leq k\leq 2^m-1,$
\begin{equation}
U_{\alpha,k}=\big\{B\subseteq V_\alpha~\big|~|B|=k,\sum_{i\in B}i=\alpha,~\text{and}~B\cap(B+\alpha)=\emptyset\big\}.
\nonumber
\end{equation}

\medskip

\begin{lem}    \label{prop_part1}
$U_{\alpha,2}$ forms a partition of $V_{\alpha}.$
\end{lem}
\begin{proof}
It immediately follows by Remark~\ref{rem_H}.
\end{proof}

\medskip

To prove the main theorem, a result has to be introduced.
\begin{rem}     \label{rem_isomorphism}
Let $A=\{0,\alpha\}.$ Then $\langle A,+\rangle$ is a subgroup of $\langle \mathbb{F}_{2^{m+1}},+\rangle.$ It is clear that the quotient group $\mathbb{F}_{2^{m+1}}/A$ is with zero $A.$ Since every nonzero element in $\mathbb{F}_{2^{m+1}}/A$ has order $2$ and $\mathbb{F}_{2^m}$ has characteristic $2,$ $\mathbb{F}_{2^{m+1}}/A$ is isomorphic to $\langle \mathbb{F}_{2^m},+\rangle$ by the fundamental theorem of finitely generated abelian groups.
\end{rem}

\medskip

Recall that for $3\leq k \leq 2^m-4$ the pair $(\mathbb{F}_{2^m}^*,W_k)$ is a $(2^m-1,k,\lambda_k)$-BIBD as shown in Theorem~\ref{thm_BIBD}. The next theorem states that the triple $(V_{\alpha},U_{\alpha,2},U_{\alpha,k})$ is a $(k,\lambda_k')$-GDD with balance parameter $\lambda_k'=2^{k-3}\lambda_k$.

\medskip

\begin{thm}     \label{thm_GDDmain}
For each $3\leq k\leq 2^m-4,$ $(V_{\alpha},U_{\alpha,2},U_{\alpha,k})$ is a $(k,\lambda_k')$-GDD with balance parameter $\lambda_k'=2^{k-3}\lambda_k.$
\end{thm}
\begin{proof}
Let $i,j$ be two distinct elements in $V_\alpha$ with $i+j\neq \alpha.$
It suffices to show that there are $2^{k-3}\lambda_k$ blocks in $U_{\alpha,k}$ that contains $i$ and $j,$ where $\lambda_k$ is the balance parameter of the BIBD $(\mathbb{F}_{2^m}^*,W_k)$ proposed in Section~\ref{subsec_proposedBIBD}.
Let $A=\{0,\alpha\}\subseteq \mathbb{F}_{2^{m+1}},$ as mentioned in Remark~\ref{rem_isomorphism}.
Then there exists an isomorphism $\psi:~\mathbb{F}_{2^{m+1}}/A\rightarrow\mathbb{F}_{2^m}.$
Moreover, let $\overline{x}=\{x,x+\alpha\}$ for $x\in V_\alpha.$
One can see that for any $B\subseteq V_\alpha,$
\begin{equation}    \label{eq_GDDeq}
\sum_{\ell\in B}\overline{\ell}=A~~~\text{if and only if}~~~
\sum_{\ell\in B}\psi(\overline{\ell})=0\in \mathbb{F}_{2^m}.
\end{equation}
Note that $\sum_{\ell\in B}\overline{\ell}=\overline{\sum_{\ell\in B}\ell}.$
Hence if $\sum_{\ell\in B}\overline{\ell}=A$ then
$\sum_{\ell\in B}\psi(\overline{\ell})=\psi(\sum_{\ell\in B}\overline{\ell})=0,$ and vice versa.

\medskip

Let $B=\{i,j,x_1,x_2,\ldots,x_{k-2}\}$ be a $k$-subset of $V_\alpha$ with $i,j\in B$ and $B\cap(B+\alpha)=\emptyset.$
On the left-hand side of~\eqref{eq_GDDeq}, if $B$ satisfies the condition $\sum_{\ell\in B}\overline{\ell}=A$, then there are $2^{k-2}$ possible choices of $k$-subset $\tilde{B}=\{i,j,y_1,y_2,\ldots,y_{k-2}\}$ of $V_\alpha$ such that $\sum_{\ell\in\tilde{B}}\ell=\alpha$ or $0$ by letting $y_h\in\{x_h,x_h+\alpha\}$ for $h=1,2,\ldots,k-2.$ Note that every $\tilde{B}$ also has the properties $i,j\in\tilde{B}$ and $\tilde{B}\cap(\tilde{B}+\alpha)=\emptyset.$ Therefore, there are $2^{k-2}/2=2^{k-3}$ possible choices of $\tilde{B}$ with $\sum_{\ell\in \tilde{B}}\ell=\alpha$ corresponding to $B.$
On the other hand, since $\psi(\overline{i})$ and $\psi(\overline{j})$ are given, by Theorem~\ref{thm_BIBD} there are $\lambda_k$ blocks for the right-hand side of~\eqref{eq_GDDeq} provided that $B$ is a $k$-subset of $V_\alpha$ with $i,j\in B$ and $B\cap(B+\alpha)=\emptyset.$
In summary, there are $2^{k-3}\lambda_k$ ways to pick a $k$-subset $B\subseteq V_\alpha$ with $i,j\in B,$ $B\cap (B+\alpha)=\emptyset,$ and $\sum_{\ell\in B}\ell=\alpha.$
Namely, the balance parameter $\lambda_k'=2^{k-3}\lambda_k.$
The result follows.
\end{proof}

\medskip

From Remark~\ref{rem_BIBD_ini_condi}, $\lambda_3'=2^0\lambda_3=1.$ Then the recurrence relations of $\lambda_k'$ is given in the following which can be attained by Theorem~\ref{thm_GDDmain} and Corollary~\ref{cor_lambda_k}.
\begin{cor}     \label{cor_lambda_k'}
For each $3\leq k\leq 2^m-4,$
\begin{equation}
\lambda_{k+1}'=\left\{\begin{array}{ll}
        \frac{2^{m+1}-2k-2}{k-1}\lambda_k', & \text{if}~k\equiv 1,3~(\text{mod}~4) \\
        \frac{2^{m+1}-2k-2}{k-1}\lambda_k'+2^{k-2}{2^{m-1}-2 \choose k/2-1}, & \text{if}~k\equiv 2~(\text{mod}~4) \\
        \frac{2^{m+1}-2k-2}{k-1}\lambda_k'-2^{k-2}{2^{m-1}-2 \choose k/2-1}, & \text{if}~k\equiv 0~(\text{mod}~4)
                      \end{array}\right.
\nonumber
\end{equation}
where $\lambda_{2^m-3}':=0.$
In one formula,
\begin{equation}
\lambda_{k+1}'=\frac{2^{m+1}-2k-2}{k-1}\lambda_k'-\cos\frac{k\pi}{2}\cdot2^{k-2}{2^{m-1}-2 \choose \lfloor k/2-1\rfloor}.
\nonumber
\end{equation}
\qed
\end{cor}

\medskip

The balance parameters of the newly proposed GDD $(V_{\alpha},U_{\alpha,2},U_{\alpha,k})$ and the previously known GDD in \cite{lcc:17} with $3\leq k\leq 7$ are compared in Table 2, where $\alpha\in\mathbb{F}_{2^{m+1}}\setminus\{0\}$ and $V_\alpha=\mathbb{F}_{2^{m+1}}\setminus\{0,\alpha\}.$
%\extrarowheight=5pt
\renewcommand{\arraystretch}{1.5}
\begin{center}
{\bf Table 2:} Comparison on balance parameters $\lambda_k'$ of GDDs for $3\leq k\leq 7.$

\bigskip
\begin{tabular}{|c|c|c|}
\hline
      & $\lambda_k'$ of Proposed GDDs & $\lambda_k'$ in \cite{lcc:17}    \\ \hline
$k=3$ & $1$ & $1$     \\ \hline
$k=4$ & $\frac{2^{m+1}-8}{2}$ & $\frac{2^{m}-8}{2}$\\ \hline
$k=5$ & $\frac{(2^{m+1}-8)(2^{m+1}-16)}{3!}$ & $\frac{(2^{m}-8)(2^{m}-16)}{3!}$\\ \hline
$k=6$ & $\frac{(2^{m+1}-8)(2^{m+1}-12)(2^{m+1}-16)}{4!}$ & $\frac{(2^{m}-8)(2^{m}-16)(2^{m}-32)}{4!}$  \\ \hline
$k=7$ & $\frac{(2^{m+1}-8)(2^{m+1}-12)(2^{2m+2}-30\cdot 2^{m+1}+284)}{5!}$ & $\frac{(2^{m}-8)(2^{m}-16)(2^{m}-32)(2^{m}-64)}{5!}$  \\ \hline
\end{tabular}\\
\end{center}

From Theorem~\ref{thm_GDDmain} and Corollary~\ref{cor_lambda_k'},
the parameters $(k,\lambda_k')$ of GDDs with small block sizes are listed below:
$(3,1),$ $(4,4),$ $(4,12),$ $(4,28),$ and $(5,64).$

\section{Conclusion}        \label{sec_conclusion}
In this paper, based on the fact that $(\mathbb{F}_{2^m}^*,W_k)$ is a $(2^m-1,k,\lambda_k)$-BIBD for $3\leq k\leq 2^m-4$ in Theorem~\ref{thm_BIBD}, we show in Theorem~\ref{thm_GDDmain} that the triple $(V_\alpha,U_{\alpha,2},U_{\alpha,k})$ is a $(k,\lambda_k')$-GDD with balance parameter $\lambda_k'=2^{k-3}\lambda_k.$ A comparison of the results in~\cite{lcc:17,j:17} and this work are listed in Table~3.
Consequently, this paper has presented a new construction of GDDs, which can be proved by a fmaily of BIBDs. One advantage of the proposed GDDs is that their block sizes are much larger than those in~\cite{lcc:17,j:17}.

\medskip

\renewcommand{\arraystretch}{1.5}
\begin{center}
{\bf Table 3:} Comparison on different constructions of GDDs.

\bigskip
\begin{tabular}{l|l|l|l}
      & Points set $X$ & Block size $k$ & Balance parameter $\lambda$   \\ \hline
GDDs in \cite{lcc:17,j:17} & $\mathbb{F}_{2^m}^*\setminus\{1\}$ & $3\leq k\leq m$ & $\prod_{i=3}^{k-1}(2^m-2^i)/(k-2)!$ \\
Proposed GDDs & $\mathbb{F}_{2^{m+1}}^*\setminus\{\alpha\}$ & $3\leq k\leq 2^m-4$ & $\lambda_k'=2^{k-3}\lambda_k$
\end{tabular}\\
\end{center}

\bigskip
\section*{Acknowledgments}
This research is supported by the Ministry of Science and Technology of Taiwan R.O.C. under the projects
MOST 103-2632-M-214-001-MY3,
MOST 104-2115-M-214-002-MY2,
MOST 106-2115-M-214-004-MY2,
MOST 106-2811-M-214-001,
and MOST 106-2221-E-214-005.

\end{document}